\documentclass[12pt]{article}
\usepackage[usenames]{color}
\usepackage{amsmath,amssymb,amsthm}
\usepackage{stmaryrd}
\usepackage{hyperref}
\usepackage[all]{xy}
\usepackage[normalem]{ulem}
\xyoption{arc}

\numberwithin{equation}{section}

\newtheorem{thm}{Theorem}[section]
\newtheorem{cor}[thm]{Corollary}
\newtheorem{prop}[thm]{Proposition}
\newtheorem{lem}[thm]{Lemma}
\theoremstyle{definition}
\newtheorem{defn}[thm]{Definition}
\theoremstyle{remark}
\newtheorem{rmk}[thm]{Remark}
\newtheorem{exam}[thm]{Example}

\newtheorem{axiom}[thm]{Axiom}

\def\co{\colon\thinspace}
\newcommand{\mb}[1]{\mathbb{#1}}
\newcommand{\mf}[1]{\mathfrak{#1}}

\newcommand{\Hom}{\ensuremath{{\rm Hom}}}

\newcommand{\colim}{\ensuremath{\mathop{\rm colim}}}
\newcommand{\hocolim}{\ensuremath{\mathop{\rm hocolim}}}
\newcommand{\holim}{\ensuremath{\mathop{\rm holim}}}

\newcommand{\overto}{\mathop\rightarrow}
\newcommand{\overfrom}{\mathop\leftarrow}

\newcommand{\Map}{\ensuremath{{\rm Map}}}

\newcommand{\Sp}{{\cal S}p}

\newcommand{\End}{{\rm End}}
\newcommand{\Aut}{{\rm Aut}}

\newcommand{\proP}{{\rm pro}\mbox{-}}

\newcommand{\fibr}{\twoheadrightarrow}
\newcommand{\cofb}{\mathop\rightarrowtail}

\newcommand{\VMap}{{\cal V}\mbox{-}{\rm Map}}
\newcommand{\VEnd}{{\cal V}\mbox{-}{\rm End}}
\newcommand{\VMon}{{\cal V}\mbox{-}{\rm Mon}}

\newcommand{\smsh}[1]{\ensuremath{\mathop{\wedge}_{#1}}}

\newcommand{\comp}[1]{\ensuremath{#1^\wedge}}

\newdir{ >}{{}*!/-10pt/@{>}}

\title{Commutative ring objects in pro-categories and generalized Moore
  spectra}
\author{Daniel G. Davis, Tyler Lawson\thanks{Partially supported by NSF
    grant 0805833 and a fellowship from the Sloan Foundation.}}

\begin{document}
\maketitle

\begin{abstract}
  We develop a rigidity criterion to show that in simplicial model
  categories with a compatible symmetric monoidal structure, operad
  structures can be automatically lifted along certain maps.  This is
  applied to obtain an unpublished result of M. J. Hopkins that
  certain towers of generalized Moore spectra, closely related to the
  $K(n)$-local sphere, are $E_\infty$-algebras in the category of
  pro-spectra.  In addition, we show that Adams resolutions
  automatically satisfy the above rigidity criterion.  In order to
  carry this out we develop the concept of an operadic model category,
  whose objects have homotopically tractable endomorphism operads.
\end{abstract}

\section{Introduction}
\label{sec:intro}

One of the canonical facts that distinguishes stable homotopy theory
from algebra is the fact that the mod $2$ Moore spectrum does not
admit a multiplication.  There are numerous consequences and
generalizations of this fact: there is no Smith-Toda complex $V(1)$ at
the prime $2$; the Smith-Toda complex $V(1)$ does not admit a
multiplication at the prime $3$; the mod $4$ Moore spectrum admits no
multiplication which is either associative or commutative; the mod $p$
Moore spectrum admits the structure of an $A(p-1)$-algebra but not an
$A(p)$-algebra; and so on.  (A discussion of the literature on
multiplicative properties of Moore spectra can be found in
\cite[A.6]{thomason-ktheory-etale}, while multiplicative properties of
$V(1)$ can be found in \cite{oka-ringcells}.  The higher structure on
Moore spectra plays an important role in \cite{schwede-rigid}.)

These facts and others form a perpetual sequence of obstructions to the
existence of strict multiplications on generalized Moore spectra, and
it appears to be the case that essentially no generalized Moore
spectrum admits the structure of an $E_\infty$-algebra.

Despite this, the goal of the current paper is to show the following:

{\it For any prime $p$ and any $n \geq 1$, let $\{M_I\}_I$ be a tower
  of generalized Moore spectra of type $n$, with homotopy limit the
  $p$-complete sphere} ({\it as in}
\cite[4.22]{hovey-strickland-moravaktheory}). {\it Then $\{M_I\}_I$
  admits the structure of an $E_\infty$-algebra in the category of
  pro-spectra.}

Roughly, the multiplicative obstructions vanish when taking the
inverse system as a whole (by analogy with the inverse system of
neighborhoods of the identity in a topological group).

This statement is due to Mike Hopkins, and it is referenced in
\cite[5.4.2]{rognes-galois}.  Mark Behrens gave a proof that the tower
admits an $H_\infty$ structure, based on Hopkins' unpublished
argument, in \cite{mark-hinfty}.  As discussed in
\cite[2.7]{ausoni-rognes-fractionfield}, Ausoni, Richter, and Rognes
worked out a version of Hopkins' statement for the pro-spectrum
$\{ku/p^\nu\}_{\nu \geq 1}$ for any prime $p$ as an object in the
category of pro-$ku$-modules. (Here, $ku$ is the connective complex
$K$-theory spectrum.)

It has been understood for some time that the $K(n)$-local category
should, in some sense, be a category with some pro-structure.  For
example, by \cite[\S 2]{hovey-splittingconjecture}, if $X$ is any
spectrum, then
\[
L_{K(n)}(X) \simeq \holim_I \,(L_nX \wedge M_I)\co
\]
the $K(n)$-localization of $X$ is the homotopy limit of the levelwise
smash product in pro-spectra of $L_nX$ with the tower $\{M_I\}_I$.  In
applications the Morava $E$-theory homology theory $E(k,
\Gamma)_\ast(-)$, as defined below, is often replaced by the more
tractable completed theory which again involves smashing with the
pro-spectrum $\{M_I\}_I$:
\[
E(k, \Gamma)^\vee_\ast(X) = \pi_\ast\bigl(L_{K(n)}(E(k, \Gamma) \wedge X)\bigr) 
\cong \pi_\ast\bigl(\holim_I\,(E(k, \Gamma) \wedge X \wedge M_I)\bigr).
\]
(For example, see \cite[\S 2]{goerss-henn-mahowald-rezk},
\cite{hovey-etheory}, and \cite{rezk-logarithmic}.)  Thus, in some
sense our goal is to establish appropriate multiplicative properties
of this procedure.

We give several applications of our results.  Let $n \geq 1$ and let
$p$ be a fixed prime.  As in \cite{rezk-notes-on-hopkins-miller}, let
${\cal{FG}}$ be the category that consists of pairs $(k, \Gamma)$,
where $k$ is any perfect field of characteristic $p$ and $\Gamma$ is a
height $n$ formal group law over $k$.  The morphisms are pairs $(r, f)
\co (k, \Gamma) \to (k', \Gamma')$, where $r \co k' \to k$ is a ring
homomorphism and $f \co \Gamma \to r^\ast(\Gamma')$ is an isomorphism
of formal group laws.

By \cite{goerss-hopkins-summary} (see also
\cite[2.7]{goerss-presheaves}), the Goerss-Hopkins-Miller Theorem
says that there is a presheaf
\[
E \co {\cal{FG}}^{op} \to \Sp_{\scriptscriptstyle{E_\infty}}, \ \ \ (k, \Gamma) \mapsto E(k, \Gamma),
\]
where $\Sp_{\scriptscriptstyle{E_\infty}}$ is the category of
commutative symmetric ring spectra and
\[
E(k, \Gamma)_\ast \cong W(k)\llbracket u_1, ..., u_{n-1} \rrbracket
[u^{\pm 1}].
\] 
Here $W(k)$ is the ring of Witt vectors of the field $k$, each $u_i$
has degree zero, and the degree of $u$ is $-2$. The $E_\infty$-algebra
$E(k, \Gamma)$ is a {\em Morava $E$-theory}, whose formal group law is
a universal deformation of $\Gamma$.  In Section \ref{sec:mainthm} we
show that each $E(k, \Gamma)$ lifts to the $E_\infty$-algebra $\{E(k,
\Gamma) \wedge M_I\}_I$ in the category of pro-spectra.

Also, in Section~\ref{sec:completions} we show
that various completions that are commonly employed in homotopy theory
also have highly multiplicative structures.  In particular, these
include classical Adams resolutions.  This has the amusing consequence
that, in homotopy theory, the completion of a commutative ring object
with respect to a very weak notion of an ideal (whose quotient is only
assumed to have a left-unital binary operation) automatically inherits
a commutative ring structure.

It should be noted that some care is required in the definition of an
$E_\infty$-algebra structure when working with pro-objects.  In this
paper, we use a definition in terms of endomorphism operads in
simplicial sets; an $E_\infty$-algebra structure is a map from an
$E_\infty$-operad to the endomorphism operad of the pro-object.  If $X
= \{x_\alpha\}_\alpha$ is a pro-object, note that this does {\em not}
define maps of pro-objects
\[
\{(E\Sigma_n)_+ \mathop\wedge_{\Sigma_n} (x_\alpha)^{\wedge
  n}\}_\alpha \to X.
\]
Roughly, the issue is that the levelwise smash product only
commutes with finite colimits in the pro-category.  In particular, it
does not represent the tensor of pro-objects with spaces
\cite[\S 4.1]{isaksen-strict}.

The starting point for the proof that Moore towers admit $E_\infty$
structures is the following algebraic observation.
\begin{prop}
  Suppose that $R$ is a commutative ring, $S$ is an $R$-module, and $e
  \in S$ is an element such that the evaluation map $\Hom_R(S,S) \to
  S$ is an isomorphism.  Then $S$ admits a unique binary
  multiplication such that $e^2 = e$, and under this multiplication
  $S$ becomes a commutative $R$-algebra with unit $e$.
\end{prop}
The proof consists of iteratively applying the adjunction
\[
\Hom_R(S^{\otimes_R n}, S) \cong \Hom_R(S^{\otimes_R (n-1)}, \Hom_R(S,S))
\]
to show that a map $S^{\otimes_R n} \to S$ is equivalent to a choice of
image of $e^{\otimes n}$; existence shows that $e\otimes e \mapsto e$
determines a binary multiplication, and uniqueness forces the
commutativity, associativity, and unitality properties.  In
particular, this applies whenever $S$ is the localization of a
quotient of $R$.  One notes that, while this proof only requires
studying maps $S^{\otimes_R n} \to S$ for $n \leq 3$, it is implicitly
an operadic proof.

This proof almost carries through when $S$ is the completion of $R$
with respect to an ideal $\mf m$.  However, in this case the topology
on $\comp{R\mspace{1mu}}_{\mf m}$ needs to be taken into account.  The
tensor product over $R$ needs to be replaced by a completed tensor
product of inverse limits of modules, which does not have a right
adjoint in general.  However, when restricted to objects which are
inverse limits of finitely presented modules, smallness gives the
completed tensor product a right adjoint (cf. \cite[B.3]{bauer-formal}). 

The paper follows roughly this line of argument, mixed with the
homotopy theory of pro-objects developed by Isaksen and Fausk
\cite{isaksen-strict,fausk-isaksen-tmodel}.

Unfortunately, the ``levelwise'' tensor product for pro-objects does
not usually have a right adjoint.  This means that the constructions of
model categories of rings and modules, from \cite{shipley-schwede} and
\cite[\S 4]{hovey-modelcategories}, do not apply in this circumstance.
Understanding these homotopical categories should be a topic worth
further investigation.

\subsection*{Outline}

We summarize the portions of this paper not previously described.  Our
work begins in Section~\ref{sec:operads} by collecting definitions and
results on the homotopy theory of operads and spaces of operad
structures on objects.  A more detailed outline is at the beginning of
that section.

In Section~\ref{sec:rigid} we flesh out the proof outlined in the
introduction.  In model categories with amenable symmetric monoidal
structure, as well as a weak variant of internal function objects,
certain ``rigid'' maps automatically allow one to lift algebra
structures uniquely from the domain to the target.

Section~\ref{sec:pro-objects} assembles together enough of the
homotopy theory of pro-objects to show that pro-dualizable objects
behave well with respect to a weak function object, allowing the
results of Section~\ref{sec:rigid} to be applied.  To obtain the main
results of this section, we place several strong assumptions on the
behavior of filtered colimits with respect to the homotopy theory.  In
particular, we require that filtered colimits represent homotopy
colimits and preserve both fibrations and finite limits.  The
main reason for restricting to this circumstance is that we need to
gain homotopical control over function spaces of the form $\lim_\beta
\colim_\alpha \Map(x_\alpha, y_\beta)$, as well as other function
objects.  (Functors such as $\Map(-,y)$ are rarely assumed to have good
behavior on towers of fibrations.)

Section~\ref{sec:modelcats} verifies all these necessary assumptions
in the case of modules over a commutative symmetric ring spectrum.
(The category of modules over a commutative differential graded
algebra is Quillen equivalent to such a category.)

The main result of the paper appears in Section~\ref{sec:mainthm},
which shows (Theorem \ref{thm:mainthm}) 
that a tower of generalized Moore spectra (constructed by
Hovey and Strickland based on previous work of Devinatz, Hopkins, and
Smith) automatically obtains an $E_\infty$-algebra structure from the
sphere.  This is then applied to show that certain chromatic
localizations of the sphere, as well as all the Morava $E$-theories 
$E(k, \Gamma)$, 
are naturally inverse limits of highly multiplicative pro-objects.

Section~\ref{sec:completions} carries out the aforementioned study of
multiplicative structure on completions.

\subsection*{Notation and assumptions}

As various model categories of pro-objects are very large and do not
come equipped with functorial factorization, there are set-theoretic
technicalities.  These include being able to define either derived
functors or a homotopy category with the same underlying object set.
We refer the reader to \cite{dwyer-hirschhorn-kan-smith-modelcats}
(e.g. \S 8) for one solution, which involves employing a larger
universe in which one constructs equivalence relations and produces
canonical definitions which can be made naturally equivalent to
constructions in the smaller universe.

For a functor $F$ with source a model category, the symbol $\mb LF$
(resp. $\mb RF$) will be used to denote the derived functor, with
domain the homotopy category of cofibrant-fibrant objects, when $F$
takes acyclic cofibrations (resp. acyclic fibrations) to weak
equivalences.  For inline operators such as $\otimes$, this will be replaced
by a superscript.  We use $[-,-]$ to denote the set of maps in the
homotopy category.

The generic symbol $\owedge$ denotes a monoidal product, while
$\otimes$ is reserved for actual tensor products and categories
tensored over simplicial sets.

For a pro-object $X$, we will often write an isomorphic diagram
using lowercase symbols $\{x_\alpha\}_\alpha$ without comment.

\subsection*{Acknowledgements}

The authors are indebted to Mark Behrens, for sharing his preprint
\cite{mark-hinfty} and for discussions related to the commutativity of
the Moore tower, and to Mike Hopkins for the ubiquitous influence of
his ideas.  The authors would also like to thank Bill Dwyer, Mark
Hovey, Birgit Richter, Marcy Robertson, John Rognes, and Jeff Smith
for discussions related to this paper.

\section{Operads}
\label{sec:operads}

In this section, we will discuss some background relating to operads
and their actions.  In essence, we would like to establish situations
where we have a model category ${\cal D}$ supporting enough structure
so that objects of ${\cal D}$ have endomorphism operads, and we would
like to ensure that these endomorphism operads are invariant under
both weak equivalences and appropriate Quillen equivalences.

This requires us to dig our way through several layers of terminology.

Endomorphism objects are functorial under isomorphisms.  Our goal is
to produce ``derived'' endomorphism objects which are functorial under
weak equivalences.  While our attention is turned towards endomorphism
operads, the methods apply when we have a very general enriching
category ${\cal V}$.  We give a functorial construction of derived
endomorphism objects in ${\cal V}$-monoids, which mostly relies on an
SM7 axiom, in Section~\ref{sec:enrich-endom}.  As a side benefit, we
obtain a definition of endomorphism objects for diagrams which will
prove necessary later.

We then turn our attention to the construction of endomorphism
operads.  By its very nature, this requires our category to carry a
symmetric monoidal structure, a model structure, and an enrichment in
spaces, and all of these must obey compatibility rules.  This presents
us with a significant number of adjectives to juggle.  We study this
compatibility in Section~\ref{sec:tens-model-categ}, finally encoding
it in the notion of an {\em operadic model category}.

The motivation for operadic model categories is the ability to extend
our enrichment, from simplicial sets under cartesian product to
symmetric sequences under the composition product.  The work of
Section~\ref{sec:enrich-endom} then produces derived endomorphism
operads.  To ensure that these constructions make sense in homotopy
theory, we show that they are invariant under an appropriate notion
of operadic Quillen equivalence.

Once this is in place, in Section~\ref{sec:spac-algebra-struct} we are
able to study a space parametrizing ${\cal O}$-algebra structures on a
fixed object, and be assured that if ${\cal O}$ is cofibrant it is an
invariant under equivalences of the homotopy type and equivalences
of the model category.

In this paper, operads are assumed to have symmetric group actions,
and no assumptions are placed on degrees $0$ or $1$.  We will write
${\rm Com}$ for the commutative operad, which is terminal among
simplicial operads and consists of a single point in each degree.

Both the definitions and the philosophy here draw heavily from
\cite{rezk-thesis}.

\subsection{Enriched endomorphisms}
\label{sec:enrich-endom}

In this section we assume that ${\cal V}$ is a monoidal category with
a model structure, and that ${\cal D}$ is a model category with a
${\cal V}$-enriched structure.  For $a, b \in {\cal D}$ we write
$\VMap_{\cal D}(a,b)$ for the enriched mapping object.

We assume that the following standard axiom holds.
\begin{axiom}[SM7]
Given a cofibration $i\co a \cofb b$ and a fibration $p\co x \fibr y$ in ${\cal
D}$, the map
\[
\VMap_{\cal D}(b,x) \to \VMap_{\cal D}(a,x) \mathop\times_{\VMap_{\cal
    D}(a,y)} \VMap_{\cal D}(b,y)
\]
is a fibration in ${\cal V}$, which is acyclic if either $i$ or $p$ is.
\end{axiom}

Write $\VMon$ for the category of monoids in ${\cal V}$.  For an
object $c \in {\cal D}$, we have a ${\cal V}$-endomorphism object
$\VEnd_{\cal D}(c) \in \VMon$.

\begin{defn}
A map in $\VMon$ is a {\em fibration} or a {\em weak
  equivalence} if the underlying map is a fibration or weak
equivalence in the category ${\cal V}$.
\end{defn}

This definition may or may not come from a model structure on the
category of ${\cal V}$-monoids.  However, under amenable circumstances
it makes sense to form the localization of $\VMon$ with respect to the
weak equivalences.


Our goal is to prove that endomorphism objects are functorial in weak
equivalences, at least on the level of homotopy categories
(Theorem~\ref{thm:homotopy-end}).  To construct this functor, it is
useful to first note that the subcategory of isomorphisms in the
homotopy category of ${\cal D}$ is naturally equivalent to a category
formed by a restricted localization.
\begin{lem}
\label{lem:invert-cofibrations}
  Let ${\cal M}$ be a model category and ${\cal A} \subset {\cal M}$
  be the subcategory of acyclic fibrations between cofibrant-fibrant
  objects.  Then the natural functor
\[
  {\cal A}^{-1} {\cal A} \to ho({\cal M})^w,
\]
from the groupoid completion to the subcategory of isomorphisms in the
homotopy category of ${\cal M}$, is fully faithful and essentially
surjective.
\end{lem}

\begin{rmk}
The dual result clearly holds for inverting acyclic cofibrations.
\end{rmk}

\begin{proof}
Any object in ${\cal M}$ is equivalent to a cofibrant-fibrant one, so
the functor is obviously essentially surjective.

Let $x, y$ be cofibrant-fibrant objects in ${\cal M}$, and consider
the obvious map $x \to x \times (\prod_f y)$, where the product is indexed
by weak equivalences $f\co x \to y$.  We can factor this map into an
acyclic cofibration $x \to \tilde x$ followed by a fibration, and for
any such weak equivalence $f$ this yields a diagram in ${\cal M}$ of
the form
\[
\xymatrix{
x \ar@{ >->}[r]^-\sim \ar[d]_=&
\tilde x \ar@{->>}[dl]_\sim \ar@{->>}[d] \ar@{->>}[dr]^\sim\\
x &
x \times (\prod_f y) \ar@{->>}[l] \ar@{->>}[r] &
y.
}
\]
This shows that ${\cal A}^{-1} {\cal A} \to ho({\cal M})^w$ is
full.  Moreover, all maps in the homotopy category are realized
in ${\cal A}^{-1} {\cal A}$ by the inverse of the map $\tilde x \fibr
x$ followed by a map $\tilde x \fibr y$.  Therefore, to complete the
proof it suffices to show that right homotopic acyclic fibrations 
$\tilde x \to y$ become equal in ${\cal A}^{-1} {\cal A}$.

Let $y \cofb z \fibr y \times y$ be a path object for $y$, with $p_0,
p_1\co z \fibr y$ the component projections (which are acyclic
fibrations).  Let $h = z \times_y z$, with the product taken over
$p_0$ on both factors, and $j_0, j_1\co h \fibr z$ the component
projections.  The maps $p_1 j_0$ and $p_1 j_1$ make the object $h$
into another path object for $y$.  However, we have an identity of
acyclic fibrations $p_0 j_0 = p_0 j_1$, and so in the category ${\cal
  A}^{-1} {\cal A}$ we have $p_1 j_0 = p_1 j_1$.
\end{proof}


\begin{defn}
For a small category $I$, the functor category ${\cal D}^I$ is
a ${\cal V}$-enriched category, with $\VMap_{{\cal D}^I}(F,G)$
described by the equalizer diagram
\[
\xymatrix@C=2.5ex{\VMap_{{\cal D}^I}(F,G) \ar[r] & \displaystyle{\prod_{i} \VMap_{\cal D}(F(i),G(i))} 
\ar@<.5ex>[r] \ar@<-.5ex>[r] & \displaystyle{\prod_{i \to j} \VMap_{\cal D}(F(i), G(j))}.
}\]
In the particular case where $I$ is the poset $\{0 < 1\}$ and ${\cal
  D}^I$ is the category of arrows $Ar({\cal D})$, we will abuse
notation by writing $\VMap_{\cal D}(f,g)$ as an enriched mapping
object between two morphisms $f$ and $g$ of ${\cal D}$.  Similarly, in
the case where $J$ is the poset $\{0 < 1 < 2\}$ and ${\cal D}^J$ is
the category of composable pairs of arrows of ${\cal D}$, for
$J$-diagrams
\[
(\cdot \overto^f \cdot \overto^{f'} \cdot) \text{ and } (\cdot \overto^g
\cdot \overto^{g'} \cdot)
\]
we will similarly write $\VMap_{\cal D}((f',f),(g',g))$ as an
enriched mapping object.
\end{defn}

\begin{rmk}
\label{rmk:diagram-sm7}
When $\VMap_{\cal D}\co {\cal D}^{op} \times {\cal D} \to {\cal V}$ preserves
limits, we can say more.  For Reedy categories $I$, the category
${\cal D}^I$ then inherits a ${\cal V}$-enriched Reedy model structure
that satisfies an SM7 axiom.  (Compare \cite{angeltveit-reedy}, which
assumes a symmetric monoidal closed structure on ${\cal V}$.)
\end{rmk}

For maps $f\co a \to b$ and $p\co x \to y$ in ${\cal D}$, the
categorical equalizer $\VMap_{\cal D}(f,p)$ can be alternatively
described in ${\cal V}$ as a fiber product
\[
\VMap_{\cal D}(a,x) \mathop\times_{\VMap_{\cal D}(a,y)} \VMap_{\cal D}(b,y).
\]
This makes the following proposition a straightforward consequence of
the SM7 axiom.

\begin{prop}[{cf. \cite[6.6]{dwyer-hess-longknots}}]
\label{prop:enriched-homotopical-map}
Suppose that in ${\cal D}$, $f\co a \to b$ is a map with cofibrant
domain and $p\co x \fibr y$ is a fibration between fibrant objects.
Then
\begin{itemize}
\item the map $\VMap_{\cal D}(f,p) \to \VMap_{\cal D}(b,y)$ is a
  fibration;
\item if $p$ is an acyclic fibration, then the map $\VMap_{\cal
  D}(f,p) \to \VMap_{\cal D}(b,y)$ is an acyclic fibration; and
\item if $p$ is an acyclic fibration and $f$ is a weak equivalence
  between cofibrant objects, then the map $\VMap_{\cal  D}(f,p) \to
  \VMap_{\cal D}(a,x)$ is a weak equivalence.
\end{itemize}
\end{prop}

Restricting to the case where $f$ and $p$ coincide, we deduce the
following consequences for endomorphisms.

\begin{cor}
\label{cor:enriched-homotopical-end}
Suppose that in ${\cal D}$, $f\co x \fibr y$ is an acyclic fibration
between fibrant objects.  If $x$ is cofibrant, then the map
$\VEnd_{\cal D}(f) \to \VEnd_{\cal D}(y)$ is an acyclic fibration in
$\VMon$, and if both $x$ and $y$ are cofibrant, then the map $\VEnd_{\cal
D}(f) \to \VEnd_{\cal D}(x)$ is a weak equivalence in $\VMon$.  
\end{cor}

Similar analysis yields the following.

\begin{prop}
\label{prop:enriched-composition}
Suppose that in ${\cal D}$, $a \xrightarrow{f} b \xrightarrow{g} c$ are
maps between cofibrant objects and $x \overset{p}{\fibr} y
\overset{q}{\fibr} z$ are fibrations with $z$ fibrant.  Then
\begin{itemize}
\item the map $\VMap_{\cal D}((g,f),(q,p)) \to
  \VMap_{\cal D}(g,q)$ is a fibration;
\item if $p$ is an acyclic fibration, then the map $\VMap_{\cal
      D}((g,f),(q,p)) \to \VMap_{\cal D}(g,q)$ is an
  acyclic fibration; and
\item if $p$ and $q$ are acyclic fibrations and $g$ is a weak equivalence, 
  then both the maps
  \begin{align*}
\VMap_{\cal D}((g,f),(q,p)) &\to \VMap_{\cal D}(f,p)\text{ and}\\
\VMap_{\cal D}((g,f),(q,p)) &\to \VMap_{\cal D}(gf,qp)
  \end{align*}
  are weak equivalences.
\end{itemize}
\end{prop}

\begin{cor}
\label{cor:enriched-end-composition}
Suppose that in ${\cal D}$, $f\co x \fibr y$ and $g\co y \fibr z$ are
acyclic fibrations between cofibrant-fibrant objects.  Then the map
$\VEnd_{\cal D}((g,f)) \to \VEnd_{\cal D}(g)$ is an acyclic
fibration in $\VMon$, and the maps
\begin{align*}
\VEnd_{\cal D}((g,f)) &\to \VEnd_{\cal D}(f) \text{ and}\\
\VEnd_{\cal D}((g,f)) &\to \VEnd_{\cal D}(gf)
\end{align*}
are weak equivalences in $\VMon$.
\end{cor}


\begin{thm}
\label{thm:homotopy-end}
Suppose that the category of ${\cal V}$-monoids has a homotopy
category $ho(\VMon)$.  Then there is a derived functor
\[
\mb R\VEnd_{\cal D}\co ho({\cal D})^w \to ho(\VMon)^w,
\]
from isomorphisms in the homotopy category of ${\cal D}$ to
isomorphisms in the homotopy category of ${\cal V}$-monoids.

This lifts the composite of the antidiagonal
\[
ho({\cal D})^w \to ho({\cal D})^{op} \times ho({\cal D})
\]
with the functor
\[
\mb R\VMap_{\cal D}\co ho({\cal D})^{op} \times ho({\cal D}) \to ho({\cal V}).
\]
The monoid $\VEnd_{\cal D}(c)$ represents the derived homotopy type in
$ho(\VMon)$ on cofibrant-fibrant objects.
\end{thm}

\begin{rmk}
  In the case of the mapping space between two objects in a model
  category, this is most easily accomplished using the Dwyer-Kan
  simplicial localization \cite{dwyer-kan-simpliciallocalization}
  (generalized in \cite{dundas-localization}).  This constructs a
  simplicially enriched category, with the correct mapping spaces,
  where the weak equivalences have become isomorphisms.

  However, as natural transformations can only be recovered in the
  simplicial localization using simplicial homotopies, study of the
  interaction between the symmetric monoidal structure and simplicial
  localization would require extra work.  The shortest path is likely
  through $\infty$-category theory, which would take us too far
  afield.
\end{rmk}

\begin{proof}[Proof of \ref{thm:homotopy-end}]
  Let ${\cal A} \subset {\cal D}$ be the category of acyclic
  fibrations between cofibrant-fibrant objects of ${\cal D}$.  By
  Lemma~\ref{lem:invert-cofibrations}, it suffices to define the
  functor ${\cal A} \to ho(\VMon)^w$.

For an acyclic fibration $f\co x \to y$ in ${\cal A}$,
Corollary~\ref{cor:enriched-homotopical-end} gives a diagram of weak
equivalences
\[
\VEnd_{\cal D}(x) \overfrom^\sim \VEnd_{\cal D}(f) \overto^\sim
\VEnd_{\cal D}(y)
\]
in $\VMon$, representing a composite map in $ho(\VMon)^w$.  
For a composition $g \circ f$ we apply
Corollary~\ref{cor:enriched-end-composition} to obtain a commutative
diagram 
\[
\xymatrix{
&\VEnd_{\cal D}((g,f)) \ar[dl] \ar[d] \ar[dr]\\
\VEnd_{\cal D}(f) \ar[d] \ar[dr] & 
\VEnd_{\cal D}(gf) \ar[dl] \ar[dr] &
\VEnd_{\cal D}(g) \ar[dl] \ar[d]\\
\VEnd_{\cal D}(x) & \VEnd_{\cal D}(y) & \VEnd_{\cal D}(z)
}
\]
of weak equivalences in $\VMon$, which shows that the resulting
assignment respects composition.
\end{proof}

We also record that by replacing ${\cal D}$ with the category
$Ar({\cal D})$ of arrows in ${\cal D}$, equipped with the projective
model structure, we obtain the following consequence of
Theorem~\ref{thm:homotopy-end}.
\begin{prop}
\label{prop:arrow-homotopy-end}
Suppose that the category of ${\cal V}$-monoids has a homotopy
category $ho(\VMon)$.  Then there is a derived functor
\[
\mb R\VEnd_{\cal D}\co ho(Ar({\cal D}))^w \to ho(\VMon)^w,
\]
from isomorphisms in the homotopy category of $Ar({\cal D})$ to
isomorphisms in the homotopy category of ${\cal V}$-monoids, together
with natural transformations
\[
\mb R\VEnd_{\cal D}(x) \overfrom \mb R\VEnd_{\cal D}(f)
\overto \mb R\VEnd_{\cal D}(y)
\]
for $f\co x \to y$.

The monoid $\VEnd_{\cal D}(f)$ represents the derived homotopy type in
$ho(\VMon)$ on fibrations between cofibrant-fibrant objects.
\end{prop}

\subsection{Tensor model categories}
\label{sec:tens-model-categ}

First, we recall interaction between a monoidal structure and a model
category structure. 

Recall that the {\em pushout product axiom} for cofibrations in a
model category with monoidal product $\owedge$ says that if $f\co x
\to y$ and $f'\co x' \to y'$ are cofibrations, then the pushout map
\[
(y \owedge x') \amalg_{x \owedge x'} (x \owedge y') \to y \owedge y'
\]
is a cofibration, and is a weak equivalence if either $f$ or $f'$ is.

\begin{defn}
\label{def:tensormodel}
  A {\em tensor model category} is a model category ${\cal D}$
  equipped with a monoidal product that
  \begin{itemize}
  \item satisfies the pushout product axiom for cofibrations,
  \item takes the product with an initial object in either variable to
    an initial object, and
  \item preserves weak equivalences when either of the inputs is
    cofibrant.
  \end{itemize}
  If ${\cal D}$ is further equipped with a symmetric tensor structure,
  ${\cal D}$ is a {\em symmetric tensor model category}.
\end{defn}

\begin{rmk}
  \label{rmk:empty-tensor}
  Note that the second component makes this more restrictive than
  \cite[12.1, 12.2]{fausk-isaksen-tmodel}.  Because the product with
  an initial object is always initial, the pushout product axiom
  implies that the monoidal product preserves cofibrant objects.
\end{rmk}

By analogy with the definition of a (lax) monoidal Quillen adjunction
\cite[3.6]{schwede-shipley-monoidalequivalences}, we have the
following.

\begin{defn}
\label{def:tensorquillen}
  Suppose that ${\cal D}$ and ${\cal D}'$ are tensor model categories.
  A {\em tensor Quillen adjunction} is a Quillen adjoint pair of functors
\[
L \colon {\cal D} \rightleftarrows {\cal D}' \colon R,
\]
  together with a lax monoidal structure on $R$, such that
  \begin{itemize}
  \item for any cofibrant objects $x,y \in {\cal D}$, the induced natural
    transformation $L(x \owedge y) \to L(x) \owedge' L(y)$ is a weak
    equivalence, and
  \item for some cofibrant replacement $\mb I_c$ of the unit $\mb I$
    of ${\cal D}$, the induced map $L(\mb I_c) \to \mb I'$ is a weak
    equivalence.
  \end{itemize}
  We refer to this as a {\em symmetric tensor Quillen adjunction} if the
  functor $R$ is lax symmetric monoidal, and a {\em tensor Quillen
  equivalence} if the underlying adjunction is a Quillen equivalence.
\end{defn}

\begin{defn}
  A {\em simplicial tensor model category} is a simplicial model
  category ${\cal D}$ equipped with a monoidal product such that
  \begin{itemize}
  \item this structure makes ${\cal D}$ into a tensor model category, 
  \item there are choices of natural isomorphisms
    \begin{align*}
K \otimes x &\cong x \owedge (K \otimes \mb I)\text{ and}\\
K \otimes x &\cong (K \otimes \mb I) \owedge x
    \end{align*}
    for $x \in {\cal D}$ and $K$ a finite simplicial set which are
    compatible with the unit isomorphism, and
  \item the functor $\Map_{\cal D}(\mb I,-)$ is a right Quillen
    functor.
  \end{itemize}
  (Here $\otimes$ denotes the tensor of objects of ${\cal D}$ with
  simplicial sets from the simplicial model structure, and $\Map_{\cal
    D}$ denotes the simplicial mapping object.)  In this case,
  we say that the monoidal structure on ${\cal D}$ is {\em compatible}
  with the simplicial model structure.

  A {\em simplicial symmetric tensor model category} is a simplicial
  tensor model category such that the composite natural isomorphism
\[
x \owedge (K \otimes \mb I) \cong K \otimes x  \cong (K \otimes
\mb I) \owedge x
\]
  is the natural symmetry isomorphism.
\end{defn}

We will freely make use of phrases such as ``simplicial (symmetric)
tensor Quillen adjunction/equivalence'' to indicate tensor Quillen
adjunctions with an appropriate lift to a lax monoidal simplicial
Quillen adjunction.

\begin{rmk}
\label{rmk:easytensoradjunction}
We note that several situations occur where the Quillen functors in
question are each the identity functor, viewed as a Quillen functor between
two distinct tensor model structures on the same monoidal category.
In this circumstance, the extra axioms for a tensor Quillen
equivalence or a simplicial symmetric Quillen equivalence are
trivially satisfied.
\end{rmk}

\begin{rmk}
  The Yoneda embedding ensures that, for any finite $K$ and $L$ and
  any $x$, $(K \times L) \otimes x \cong K \otimes (L \otimes x)$.
  Compatibility then implies that this is isomorphic to $(K \otimes
  \mb I) \owedge (L \otimes \mb I) \owedge x$.  These can be used to
  obtain well-behaved maps
\[
\Map_{\cal D}(x,y) \times \Map_{\cal D}(x',y') \to \Map_{\cal D}(x
\owedge x', y \owedge y').
\]
  If the tensor structure is symmetric, this map is equivariant with
  respect to the symmetry isomorphisms.
\end{rmk}

\begin{rmk}
  \label{rmk:cofibrant-unit}
  The following are equivalent.
  \begin{enumerate}
  \item The unit object $\mb I$ is cofibrant in ${\cal D}$.
  \item The functor $\Map_{\cal D}(\mb I, -)$, from ${\cal D}$ to
    simplicial sets, is a right Quillen functor.
  \item The functor $(-) \otimes  \mb I$, from simplicial sets to
    ${\cal D}$, is a left Quillen functor.
  \item The functor $(-) \otimes \mb I$, from simplicial sets to
    ${\cal D}$, preserves cofibrations.
  \end{enumerate}
  Evidently each implies the next.  To complete the equivalence, we
  take the cofibration $\emptyset \cofb *$ and tensor with $\mb I$,
  which (again, checking the Yoneda embedding) is naturally isomorphic
  to the map from an initial object of ${\cal D}$ to $\mb I$.

  It is unsatisfying to make the assumption that the unit is
  cofibrant, but it will ensure homotopical control on endomorphism
  operads.  It may be dropped if we are willing to define operads
  without an object parametrizing $0$-ary operations, but this
  significantly complicates the proof of
  Proposition~\ref{prop:einfty-unit}.
\end{rmk}

The hypotheses of a simplicial tensor model category are designed to
ensure that the monoidal structure can produce a reasonably-behaved
multicategorical enrichment, and hence reasonably-behaved
endomorphism operads.

For the sake of brevity in this paper we employ the following
shorthand, with the implicit understanding that it demonstrates a
prejudice towards simplicial sets.

\begin{defn}
An {\em operadic model category} is a simplicial symmetric tensor
model category.  An {\em operadic Quillen adjunction} is a simplicial
symmetric tensor Quillen adjunction, and if the underlying adjunction
is a Quillen equivalence we refer to it as an {\em operadic Quillen
equivalence}.
\end{defn}

We now relate these to operads in the ordinary sense.

Recall that a {\em symmetric sequence} is a collection of simplicial
sets $\{X(n)\}_{n \geq 0}$ equipped with actions of the symmetric
groups $\Sigma_n$.  There is a model structure on symmetric sequences
whose fibrations and weak equivalences are collections of equivariant
maps $X(n) \to Y(n)$ which satisfy these properties levelwise
(ignoring the action of the symmetric group).  The category of
symmetric sequences has a (non-symmetric) monoidal structure $\circ$,
the {\em composition product}, whose algebras are operads
\cite{markl-shnider-stasheff-operads}.

The main reason for introducing the concept of an operadic model
category is the following proposition.

\begin{prop}
Let ${\cal V}$ be the category of symmetric sequences of simplicial
sets.  Then for an operadic model category ${\cal D}$, the definition
\[
\VMap_{\cal D}(x,y) = \{\Map_{\cal D}(x^{\owedge n}, y)\}_n
\]
makes ${\cal D}$ into a ${\cal V}$-enriched category satisfying the
SM7 axiom.
\end{prop}

\begin{proof}
This is a straightforward consequence of the structure on ${\cal
D}$, though it requires Remarks~\ref{rmk:empty-tensor} and
\ref{rmk:cofibrant-unit}.
\end{proof}

\begin{rmk}
In particular, for a map $f\co x \to y$ in ${\cal D}$, the
endomorphism operad $\VEnd_{\cal D}(f)$ is the symmetric sequence
which, in degree $n$, is the pullback of the diagram
\[
\Map_{\cal D}(x^{\owedge n}, x) \to 
\Map_{\cal D}(x^{\owedge n}, y) \leftarrow
\Map_{\cal D}(y^{\owedge n}, y).
\]
\end{rmk}

\begin{rmk}
  While operadic model categories have natural ${\cal V}$-enrichments,
  operadic Quillen adjunctions do not automatically yield ${\cal
    V}$-enriched adjunctions, except in a homotopical sense, unless
  both adjoints are strong monoidal.
\end{rmk}

\subsection{Model structures on operads}

The model structure on the category of symmetric sequences lifts to
one on the category of operads in simplicial sets, with fibrations and
weak equivalences defined levelwise
\cite[3.3.1]{berger-moerdijk-operads}.

This extends to a simplicial model structure.  This exact statement
does not appear to be in the immediately available literature.
However, one can obtain it using either of the following approaches.
\begin{itemize}
\item Rezk's thesis constructs a simplicial model structure on operads
  with weak equivalences and fibrations defined levelwise under an
  {\em equivariant} model structure \cite[3.2.11]{rezk-thesis},
  extending a simplicial model structure on symmetric sequences.  The
  method of proof extends to the Berger-Moerdijk model structure, with
  weak equivalences and fibrations defined to be ordinary {\em
    nonequivariant} weak equivalences and fibrations, by discarding
  some of the generating cofibrations and generating acyclic
  cofibrations.  (This does not alter Rezk's Proposition~3.1.5, the
  main technical tool for proving the result, which uses the existence
  of a functorial levelwise fibrant replacement for simplicial operads
  as in \cite[B2]{schwede-gammaspace}.)
\item Alternatively, we can use the fact that operads can be expressed
  algebraically.  There is a functor which takes an $\mb N$-graded set
  $X = \{X_n\}$ and produces the free operad $\mb O(X)$ on $X$ (which
  can be expressed in terms of rooted trees with nodes appropriately
  labelled by elements of $X$).  The functor $\mb O$ is a monad on
  graded sets whose algebras are discrete operads.  It also commutes with
  filtered colimits, which makes it a {\em multisorted theory} in the
  terminology of \cite{rezk-propermodel}.  One can then apply
  \cite[Theorem 7.1]{rezk-propermodel} to obtain the desired simplicial
  model structure on the category of simplicial ${\mb
  O}$-algebras---i.e., operads in simplicial sets.
\end{itemize}

\subsection{Spaces of algebra structures}
\label{sec:spac-algebra-struct}

In this section, we assume that ${\cal D}$ is an operadic model
category, viewed as a model category enriched in symmetric sequences
of simplicial sets.

From this point forward, we will drop the enriching category from some
of the notation as follows.  For an object $x \in {\cal D}$, the 
endomorphism operad $\End_{\cal D}(x)$ is the symmetric sequence
which, in degree $n$, is the simplicial set $\Map_{\cal D}(x^{\owedge
  n}, x)$.  Similarly, for a map $f\co x \to y$ in ${\cal D}$, we have
the endomorphism operad $\End_{\cal D}(f)$.

\begin{defn}
  For a cofibrant operad ${\cal O}$ and a cofibrant-fibrant object $x
  \in {\cal D}$, the {\em space of ${\cal O}$-algebra structures on
    $x$} is the space of operad maps
\[
\Map_{operad}({\cal O}, \End_{\cal D}(x)).
\]
  For a map $\eta\co x \to y$ between cofibrant-fibrant objects in
  ${\cal D}$, the {\em space of ${\cal O}$-algebra structures on $\eta$}
  is the space of operad maps
\[
\Map_{operad}({\cal O}, \End_{\cal D}(\eta)).
\]
  Equivalently, this is the space of pairs of ${\cal O}$-algebra
  structures on $x$ and $y$ making $\eta$ into a map of ${\cal
  O}$-algebras.
\end{defn}

\begin{cor}
  If ${\cal O}$ is a cofibrant operad, a weak equivalence $f\co x \to
  y$ between cofibrant-fibrant objects in ${\cal D}$ determines an
  isomorphism in the homotopy category between the spaces of ${\cal
    O}$-algebra structures on $x$ and $y$.
\end{cor}

\begin{proof}
  By Theorem~\ref{thm:homotopy-end}, we find that the operads
  $\End_{\cal D}(x)$ and $\End_{\cal D}(y)$ are canonically equivalent
  in the homotopy category of operads, and the spaces of maps from
  ${\cal O}$ are equivalent.
\end{proof}

\begin{prop}
\label{prop:einfty-unit}
Let $f\co \mb I \cofb {\mb I}_f$ be a fibrant replacement for the unit
object of ${\cal D}$.  Then the space of $E_\infty$-algebra structures
on ${\mb I}_f$ compatible with the multiplication on $\mb I$ is
contractible.
\end{prop}

\begin{proof}
The map $\mb I \cofb \mb I_f$ is an acyclic cofibration between
  cofibrant objects.  The enrichment of the opposite category ${\cal
    D}^{op}$ gives rise to a dual formulation of
  Corollary~\ref{cor:enriched-homotopical-end}, and specifically
  implies that the map $\End_{\cal D}(f) \to \End_{\cal D}(\mb I)$ is
  an acyclic fibration.

  Let ${\cal E}$ be a cofibrant $E_\infty$-operad (a cofibrant
  replacement for ${\rm Com}$), and fix the map ${\cal E} \to {\rm
    Com} \to \End_{\cal D}(\mb I)$ coming from $\mb I$ being the unit.
  Then the space of lifts in the diagram
\[
\xymatrix{
&& \End_{\cal D}(f) \ar@{->>}^\sim[d] \ar[r] & \End_{\cal D}(\mb I_f)\\
{\cal E} \ar@{.>}[urr] \ar[rr] &&\End_{\cal D}(\mb I)
}
\]
is contractible.  However, via the map ${\cal E} \to \End_{\cal D}(\mb
I_f)$, these lifts precisely parametrize $E_\infty$-algebra structures
on $\mb I_f$ which are compatible with the multiplication on $\mb I$.
\end{proof}

Finally, we note that endomorphism operads are invariant under certain
Quillen equivalences.

\begin{prop}
\label{prop:operadicendomorphisms}
  Suppose that $L \colon {\cal D} \rightleftarrows {\cal D}' \colon R$ 
  is an operadic Quillen adjunction.  Then for any cofibrant-fibrant
objects $y \in {\cal D}$ and $x \in {\cal D}'$ with an equivalence
$f\co y \to Rx$, there is a map in the homotopy category of operads
from $\End_{{\cal D}'}(x)$ to $\End_{\cal D}(y)$.

If, in addition, this adjunction is an operadic Quillen equivalence,
this map is an isomorphism in the homotopy category of operads.
\end{prop}

\begin{proof}
Using Proposition~\ref{prop:enriched-homotopical-map}, we may assume
that the equivalence $f$ is an acyclic fibration.

Since $R$ has a simplicial lift which is lax symmetric monoidal, we
obtain a natural map
\[
\End_{\cal D'}(x) \to \End_{\cal D}(Rx)
\]
of operads.  By Corollary~\ref{cor:enriched-homotopical-end}, we
have an acyclic fibration $\End_{\cal D}(f) \overset{\sim}{\fibr} \End_{\cal
D}(Rx)$.  The composite
\[
\End_{\cal D'}(x) \to \End_{\cal D}(Rx) \mathop\twoheadleftarrow^\sim \End_{\cal
  D}(f) \to \End_{\cal D}(y)
\]
provides the desired map in the homotopy category of operads.

Now we further assume that the adjunction is an operadic Quillen
equivalence. Form the pullback
\[
{\cal O} = \End_{\cal D}(f) \mathop\times_{\End_{\cal D} (Rx)} \End_{\cal D'}(x).
\]
The map ${\cal O} \to \End_{\cal D'}(x)$ is a weak equivalence.
To complete the proof it therefore suffices to show that the map
${\cal O} \to \End_{\cal D}(y)$ is a weak equivalence. 

In degree $n$, ${\cal O}$ is the pullback of the diagram
\[
\Map_{\cal D}(x^{\owedge n}, x) \fibr \Map_{\cal D}(x^{\owedge n}, Ry)
\leftarrow \Map_{\cal D'}(y^{\owedge n},y),
\]
so it suffices to show that the right-hand map is an equivalence.
However, this map is the composite
\[
\Map_{\cal D'}(y^{\owedge n}, y) \to \Map_{\cal D'}((Lx)^{\owedge n},
y) \to \Map_{\cal D'}(L(x^{\owedge n}), y) \cong \Map_{\cal
  D}(x^{\owedge n}, Ry).
\]
The first of these maps is an equivalence because $y$ is
cofibrant-fibrant and $Lx$ is cofibrant, while the second is an
equivalence because $L(x^{\owedge n}) \to (Lx)^{\owedge n}$ is an
equivalence in ${\cal D}'$ between cofibrant objects (by definition of
a tensor Quillen adjunction).
\end{proof}

\section{Algebra structures on rigid objects}
\label{sec:rigid}

In this section we assume that ${\cal D}$ is an operadic model
category.  Our goal is to prove a rigidity result
(Theorem~\ref{thm:rigidlifting}) allowing us to lift algebra
structures, as mentioned in the introduction.  In order for this to be
ultimately applicable to pro-objects, we will first need to develop a
theory which applies when the tensor structure carries something
weaker than a right adjoint.

\subsection{Weak function objects}
\label{sec:weak-funct-objects}

\begin{defn}
\label{def:weakfunction}
A {\em weak function object} for the homotopy category $ho({\cal D})$ is a
functor
\[
F^{weak}(-,-)\co ho({\cal D})^{op} \times ho({\cal D}) \to ho({\cal
  D})
\]
equipped with a natural transformation of functors
\[
\mb R\Map_{\cal D}(x \owedge^{\mb L} y, z) \to \mb R\Map_{\cal D}(x,F^{weak}(y,z)) 
\]
in the homotopy category of spaces.  For specific $x$, $y$, and $z$
such that this map is an isomorphism in the homotopy category of
spaces, we will say that the weak function object {\em provides an
  adjoint} for maps $x \owedge^{\mb L} y \to z$.
\end{defn}

\begin{exam}
  Suppose the tensor model category $\mathcal{D}$ is closed, and use
  $F_{\cal D}(x,y)$ to denote the internal function object in
  $\mathcal{D}$.  Then for any $x$ cofibrant in $\mathcal{D}$, the
  functor $(-) \owedge x \co \mathcal{D} \to \mathcal{D}$ is a left
  Quillen functor, the adjoint $F_{\cal D}(x,-) \co \mathcal{D} \to
  \mathcal{D}$ is the corresponding right Quillen functor, and these
  determine an adjunction on the homotopy category. It follows that
  given arbitrary $x$ and $y$ in $ho(\mathcal{D})$, if $F^{weak}(x,y)$
  is defined to be the image of $F_{\cal D}(x_c, y_f)$, where $x_c$ and
  $y_f$ are cofibrant and fibrant representatives of $x$ and $y$
  respectively, then $\mathcal{D}$ has a weak function object that
  provides an adjoint for $x \owedge^\mathbb{L} y \to z$ for all $x,
  y, z \in ho(\mathcal{D})$.
\end{exam}

\begin{rmk}
We have the following consequences of Definition~\ref{def:weakfunction}.
  \begin{itemize}
  \item Substituting $x = \mb I$, we obtain a natural transformation
\[
\mb R\Map_{\cal D}(y, z) \to \mb R\Map_{\cal D}(\mb I,F^{weak}(y,z)).
\]
  \item Substituting $x = z$ and $y = \mb I$, the image of the natural
    isomorphism $x \owedge^{\mb L} \mb I \to x$ is a homotopy class of
    map $x \to F^{weak}(\mb I,x)$.  If this is a natural isomorphism,
    we refer to the weak function object as {\em unital}.
  \item Given a map $f\co x \to x'$ between objects, the natural
    transformation of functors
    \[
    F^{weak}(x',-) \to F^{weak}(x,-)
    \]
    will be referred to as the {\em map induced by} $f$ and denoted by
    $f^*$.  Similarly, the natural transformation
    \[
    F^{weak}(-,x) \to F^{weak}(-,x')
    \]
    will be denoted by $f_*$.
  \item If the weak function object provides an adjoint for
    $F^{weak}(y,z) \owedge^{\mb L} y \to z$, the identity self-map of
    $F^{weak}(y,z)$ lifts to a natural evaluation map $F^{weak}(y,z)
    \owedge^{\mb L} y \to z$ in the homotopy category of ${\cal D}$.
  \end{itemize}
\end{rmk}

\subsection{Rigid objects}

\begin{defn}
\label{def:rigid}
  Suppose that ${\cal D}$ has a weak function object.  A
  map $\eta\co x \to y$ in the homotopy category of ${\cal D}$ is {\em
    rigid} if the map
\[
\eta^*\co F^{weak}(y,y) \to F^{weak}(x,y)
\]
  is a weak equivalence.
\end{defn}

\begin{thm}
\label{thm:rigidlifting}
  Suppose that $\eta\co x \to y$ is a rigid map in $ho({\cal D})$.  In
  addition, suppose that for any $n, m \geq 0$, the weak function
  object provides adjoints for
\begin{align*}
(x^{\owedge^{\mb L} n} \owedge^{\mb L} y^{\owedge^{\mb L} m})
\owedge^{\mb L} x &\to y\text{ and}\\
(x^{\owedge^{\mb L} n} \owedge^{\mb L} y^{\owedge^{\mb L} m})
\owedge^{\mb L} y &\to y.
\end{align*}
  Then the map of operads $\mb R\End_{\cal D}(\eta) \to \mb
  R\End_{\cal D}(x)$ is an equivalence.

  In particular, if ${\cal O}$ is a cofibrant operad and $x$ is
  equipped with a homotopy class of ${\cal O}$-algebra structure
  $\theta\co {\cal O} \to \mb R\End_{\cal D}(x)$, the homotopy fiber over
  $\theta$ of the map
\[
\mb R\Map_{operad}({\cal O},\mb R\End_{\cal D}(\eta)) 
\to \mb R\Map_{operad}({\cal O}, \mb R\End_{\cal D}(x))
\]
is contractible. 
\end{thm}

\begin{proof}
  By Corollary~\ref{cor:enriched-homotopical-end} we can represent
  $\eta$ by a fibration $\eta\co x \fibr y$ between cofibrant-fibrant
  objects in ${\cal D}$.  This implies that the iterated tensor powers
  $x^{\owedge n}$ and $y^{\owedge n}$ are also cofibrant by
  Remark~\ref{rmk:empty-tensor} (with $n=0$ true by assumption on
  ${\cal D}$).

  We then apply the rigidity of $\eta$ and the adjoints provided by
  the weak function object to find that in the diagram of spaces
\[
\xymatrix{
\Map_{\cal D}(x^{\owedge n} \owedge y^{\owedge (m+1)},y) \ar[r]^-{(1 \owedge
  \eta \owedge 1)^*} \ar[d]_\sim &
\Map_{\cal D}(x^{\owedge (n+1)} \owedge y^{\owedge m},y) \ar[d]^\sim\\
\mb R\Map_{\cal D}(x^{\owedge n} \owedge y^{\owedge m}, F^{weak}(y,y))
\ar[r]^-\sim&
\mb R\Map_{\cal D}(x^{\owedge n} \owedge y^{\owedge m}, F^{weak}(x,y)),\\
}
\]
  the top map is a weak equivalence.  In particular, we find by
  induction that
\[
(\eta^{\owedge n})^*\co \Map_{\cal D}(y^{\owedge n},y) \to \Map_{\cal D}(x^{\owedge n},y)
\]
  is a weak equivalence for all $n$.  (Moreover, the source and target
  of $(\eta^{\owedge n})^*$ represent derived function spaces.)

  The endomorphism operad $\End_{\cal D}(\eta)$ is the symmetric
  sequence which, in degree $n$, is the pullback of the diagram
\[
\Map_{\cal D}(x^{\owedge n}, x) \fibr
\Map_{\cal D}(x^{\owedge n},y)
\overfrom^{\sim} \Map_{\cal D}(y^{\owedge n},y).
\]
  In each degree $\End_{\cal D}(\eta)(n)$ is a homotopy pullback of the
  above diagram because one of the maps is a fibration.  As the other
  map in this diagram is an equivalence and simplicial sets are right
  proper, we find that the ``forgetful'' map of operads $\End_{\cal
  D}(\eta) \to \End_{\cal D}(x)$ is a levelwise weak equivalence as
  desired.
  
  For any ${\cal O}$-algebra structure $\theta\co {\cal O}
  \to \End_{\cal D}(x)$, the weak equivalence $\End_{\cal D}(\eta)
  \to \End_{\cal D}(x)$ implies that the homotopy fiber over $\theta$
  is contractible, or equivalently that the space of homotopy lifts in
  the diagram
\[
\xymatrix{
  & \End_{\cal D}(\eta) \ar[d]^\sim\\
  {\cal O} \ar[r] \ar@{.>}[ur]
  & \End_{\cal D}(x)
}
\]
  is contractible as well.
\end{proof}

As a consequence of Proposition~\ref{prop:einfty-unit}, we have the
following.
\begin{cor}
  \label{cor:rigid-einfty}
  Suppose $\eta\co \mb I \to y$ is a rigid map in $ho({\cal D})$, and
  that for any $n \geq 0$ the weak function object provides adjoints
  for the maps 
  \begin{align*}
    y^{\owedge^{\mb L} n} \owedge^{\mb L} y &\to y \text{ and}\\
    y^{\owedge^{\mb L} n} \owedge^{\mb L} \mb I &\to y.
  \end{align*}
  For any cofibrant $E_\infty$-operad ${\cal E}$, the space of
  extensions to an action of ${\cal E}$ on $y$ making $\eta$ into an
  $E_\infty$-algebra map is contractible.
\end{cor}

\section{Pro-objects}
\label{sec:pro-objects}

We first recall the basics on pro-objects in a category ${\cal C}$.

\begin{defn}
  For a category ${\cal C}$, the {\em pro-category} $\proP{\cal C}$ is
  the category of cofiltered diagrams $X = \{x_\alpha\}_\alpha$ of objects of
  ${\cal C}$, with maps $X \to Y = \{y_\beta\}_\beta$ defined by 
\[
\Hom_{\proP{\cal C}}(X, Y) = \lim_\beta
\colim_\alpha \Hom_{\cal C}(x_\alpha, y_\beta).
\]
  For two cofiltered systems $X$ and $Y$ indexed by the same category, a
  {\em level map} $X \to Y$ is a natural transformation of diagrams; any map is
  isomorphic in the pro-category to a level map
  \cite[Appendix~3.2]{artin-mazur-etale}.

  A map $X \to Y$ of pro-objects satisfies a property {\em essentially
  levelwise} if it is isomorphic to a level map such that each
  component $x_\alpha \to y_\alpha$ satisfies this property.
\end{defn}

\begin{rmk}
\label{rmk:replace-index}
For any cofiltered index category $J$, there exists a final map $I \to
J$ where $I$ is a cofinite directed set
\cite[2.1.6]{edwards-hastings-cech}. This allows us to replace any
pro-object by an isomorphic pro-object indexed on a cofinite directed
set.
\end{rmk}

\subsection{Model structures}
We now recall the strict model structure on pro-objects from
\cite{isaksen-strict}.
\begin{defn}[{\cite[3.1, 4.1, 4.2]{isaksen-strict}}]
  Suppose ${\cal C}$ is a model category.  A map $X \to Y$ in
  $\proP{\cal C}$ is:
  \begin{itemize}
  \item a {\em strict weak equivalence} if it is an essentially
    levelwise weak equivalence;
  \item a {\em strict cofibration} if it is an essentially levelwise
    cofibration;
  \item a {\em special fibration} if it is isomorphic to a level map
    $\{x_\alpha \to y_\alpha\}_\alpha$ indexed by a cofinite directed set
    such that, for all $\alpha$, the relative matching map
    \[
    x_\alpha \to (\lim_{\beta < \alpha} x_\beta) \times_{\lim_{\beta <
        \alpha} y_\beta} y_\alpha
    \]
    is a fibration;
  \item a {\em strict fibration} if it is a retract of a special
    fibration.
  \end{itemize}
\end{defn}

\begin{rmk}
  \label{rmk:injectivemodel}
  If $I$ is cofinite directed, the category of $I$-diagrams admits an
  injective model structure (equivalently, a Reedy model structure)
  where weak equivalences and cofibrations are defined levelwise
  \cite[5.1.3]{hovey-modelcategories}, \cite[\S
  3.2]{edwards-hastings-cech}.  In this structure, the fibrations are
  precisely those maps satisfying the condition in the definition of a
  special fibration, and fibrant objects are also levelwise fibrant.

  By Remark~\ref{rmk:replace-index}, every pro-object $X$ can be
  reindexed to an isomorphic pro-object $X'$ indexed by a cofinite
  directed set.  There is then a levelwise acyclic cofibration $X' \to
  X_f$ where $X_f$ is an injective fibrant diagram, and hence
  represents a strict fibrant replacement; in addition, there is an
  injective fibration $X_c \to X'$ which is a levelwise weak
  equivalence, where $X_c$ is levelwise cofibrant, which represents a
  strict cofibrant replacement.
\end{rmk}

\begin{thm}[{\cite[4.15]{isaksen-strict}}]
  If ${\cal C}$ is a proper model category, then the classes of strict
  weak equivalences, strict cofibrations, and strict fibrations define
  a proper model structure on $\proP{\cal C}$.
\end{thm}

If ${\cal C}$ has a simplicial enrichment, we can extend this notion
to the category $\proP{\cal C}$.

\begin{defn}[{\cite[\S 4.1]{isaksen-strict}}]
  Let ${\cal C}$ be a simplicial model category.  For objects $X$ and $Y$
  in $\proP{\cal C}$, we define the mapping simplicial set by
\[
\Map_{\proP{\cal C}}(X, Y) = \lim_\beta \colim_\alpha \Map_{\cal
  C}(x_\alpha, y_\beta).
\]
  For $X \in \proP{\cal C}$, the tensor and cotensor with a finite
  simplicial set $K$ are defined levelwise, and for arbitrary $K$ using
  limits and colimits in the pro-category.
\end{defn}

\begin{rmk}
  As stated in the introduction, it is important to remember that
  limits and colimits of pro-objects cannot be formed levelwise (even
  for systems of level maps).  In particular, for infinite complexes
  $K$ the levelwise tensor and cotensor generally do not represent the
  tensor and cotensor in $\proP{\cal C}$.
\end{rmk}

\begin{thm}[{\cite[4.17]{isaksen-strict}}]
  If ${\cal C}$ is a proper simplicial model category, then the
  strict model structure on $\proP{\cal C}$ is also a simplicial model
  structure.
\end{thm}

\begin{thm}[{\cite[6.4]{fausk-isaksen-modelstruct}}]
\label{thm:quillenadjunction}
  Suppose $L \colon {\cal C} \rightleftarrows {\cal D} \colon R$
  is a Quillen adjunction between proper model categories.  Then the
  induced adjunction of pro-categories
$\{L\} \colon \proP{\cal C} \rightleftarrows \proP{\cal D} \colon \{R\}$
  is a Quillen adjunction.  If the original adjunction is a Quillen
  equivalence, then so is the adjunction on pro-categories.
\end{thm}

\begin{cor}
\label{cor:simplicialquillenadjunction}
  Suppose $L \colon {\cal C} \rightleftarrows {\cal D} \colon R$
  is a simplicial Quillen adjunction between proper simplicial model
  categories.  Then the induced Quillen adjunction
$\{L\} \colon \proP{\cal C} \rightleftarrows \proP{\cal D} \colon
\{R\}$
lifts to a simplicial Quillen adjunction between the pro-categories.
\end{cor}

\begin{proof}
  Applying $\lim_\beta \colim_\alpha$ to the natural isomorphism
\[
\Map_{\cal C}(x_\alpha, Ry_\beta) \cong \Map_{\cal D}(Lx_\alpha,
y_\beta)
\]
  extends the adjunction to a simplicial adjunction.
\end{proof}

\begin{prop}
\label{prop:maphomotopytype}
  Suppose ${\cal C}$ is a proper simplicial model category.  For
  levelwise cofibrant $X$ and levelwise fibrant $Y$ in $\proP{\cal C}$
  with strict fibrant replacement $Y_f$, the homotopically correct mapping 
  simplicial set $\Map_{\proP{\cal C}}(X,Y_f)$ is a natural
  representative for the homotopy type
\[
\holim_\beta \hocolim_\alpha \mb R\Map_{\cal C}(x_\alpha, y_\beta).
\]
\end{prop}

\begin{proof}
  Since $X$ is strict cofibrant, {\cite[5.3]{fausk-isaksen-tmodel}}
  shows that $\Map_{\proP{\cal C}}(X, Y_f)$ is weakly equivalent to
\[
\holim_\beta \colim_\alpha \Map_{\cal C}(x_\alpha, y_\beta).
\]
Because $X$ is levelwise cofibrant and $Y$ is levelwise fibrant, the
mapping spaces $\Map_{\cal C}(x_\alpha, y_\beta)$ are representatives
for the derived mapping spaces.  Finally, in simplicial sets, filtered
colimits are always representatives for homotopy colimits because
filtered colimits preserve weak equivalences.
\end{proof}

\subsection{Tensor structures}

\begin{defn}[{\cite[\S 11]{fausk-isaksen-tmodel}}]
  Suppose ${\cal C}$ has a monoidal operation $\owedge$ with unit $\mb
  I$.  The {\em levelwise monoidal structure} on $\proP{\cal C}$ is
  defined so that for $X,Y \in \proP{\cal C}$ indexed by $I$ and $J$
  respectively, the tensor $X \owedge Y$ is the pro-object $\{x_\alpha
  \owedge y_\beta\}_{\alpha \times \beta}$ indexed by $I \times J$.
  The unit is the constant pro-object $\mb I$.
\end{defn}

\begin{rmk}
  Note that this tensor structure on $\proP{\cal C}$ is almost never
  closed, even when the tensor structure on $\mathcal{C}$ is, as the
  levelwise tensor usually does not commute with colimits (including
  infinite coproducts) in either variable.  However, the constant
  pro-object $\{\emptyset\}$ is initial in the pro-category, and is
  preserved by the levelwise tensor product if it is preserved by
  $\owedge$.
\end{rmk}

\begin{prop}[{\cite[12.7, 12.3]{fausk-isaksen-tmodel}}]
  If ${\cal C}$ is a proper tensor model category, then the strict model
  structure on $\proP{\cal C}$ is also a tensor model category under
  the levelwise tensor structure.

  If, in addition, ${\cal C}$ is an operadic model category, the
  levelwise tensor structure on $\proP{\cal C}$ makes $\proP{\cal C}$
  into an operadic model category.
\end{prop}

\begin{prop}
\label{prop:tensorquillenadjunction}
  Suppose $L \colon {\cal C} \rightleftarrows {\cal D} \colon R$ is a
  tensor Quillen adjunction between proper tensor model categories.
  Then the induced adjunction $\{L\}\colon \proP{\cal C}
  \rightleftarrows \proP{\cal D} \colon \{R\}$ is a tensor Quillen
  adjunction, which is symmetric if the original Quillen adjunction
  is.
\end{prop}

\begin{proof}
  By Theorem~\ref{thm:quillenadjunction}, the pair $\{L\}$ and $\{R\}$
  form a Quillen adjunction. For pro-objects $X$ and $Y$, the maps
  $Rx_\alpha \owedge Ry_\beta \to R(x_\alpha \owedge y_\beta)$
  assemble levelwise to a natural lax monoidal structure for the
  functor $\{R\}$ on pro-objects, and the induced natural
  transformations for $\{L\}$ are also computed levelwise.  If $X$ and
  $Y$ are cofibrant objects of $\proP{\cal C}$, we may choose
  levelwise cofibrant models which make the conditions of
  Definition~\ref{def:tensorquillen} immediate.
\end{proof}

Combining this with Corollary~\ref{cor:simplicialquillenadjunction},
we obtain the following.

\begin{cor}
\label{cor:operadicquillenadjunction}
  Suppose $L \colon {\cal C} \rightleftarrows {\cal D} \colon R$
  is an operadic Quillen adjunction between proper operadic model
  categories.  Then the induced Quillen adjunction $\{L\}\colon\proP{\cal
  C} \rightleftarrows \proP{\cal D} \colon \{R\}$ is an operadic
  Quillen adjunction.
\end{cor}

\subsection{Function objects}
\label{sec:function-objects}

For the remainder of Section~\ref{sec:pro-objects}, we will suppose
that ${\cal C}$ is a proper operadic model category whose monoidal
product is symmetric monoidal closed.  Specifically, there is a
cofibrant unit object $\mb I$, and for objects $x, y \in {\cal C}$ we
have a product $x \owedge y$ and an internal function object $F_{\cal
  C}(x,y)$.

\begin{defn}[{\cite[9.14]{fausk-equivariant}}]
  There is a functor 
  \[
  F_{\proP{\cal C}}\co (\proP{\cal C})^{op} \times \proP{\cal C} \to \proP{\cal C}
  \] 
  defined by 
  \[
  F_{\proP{\cal C}}(X,Y) = \{\colim_\alpha F_{\cal C}(x_\alpha, y_\beta)\}_\beta,
  \]
equipped with a natural transformation
  \[
  \Map_{\proP{\cal C}}(X \owedge Y, Z) \to
  \Map_{\proP{\cal C}}(X, F_{\proP{\cal C}}(Y,Z))
  \]
  given by the composite
  \begin{align*}
    \lim_\gamma \colim_{\alpha \times \beta} & \, \Map_{\cal
      C}(x_\alpha \owedge y_\beta, z_\gamma) \xrightarrow{\cong}
    \lim_\gamma \colim_\alpha \colim_\beta
    \Map_{\cal C}(x_\alpha, F_{\cal C}(y_\beta, z_\gamma))\\
    & \to \lim_\gamma \colim_\alpha \Map_{\cal C}(x_\alpha,
    \colim_\beta F_{\cal C}(y_\beta, z_\gamma)).
  \end{align*}
\end{defn}

\begin{rmk}
In particular, the case $X = \{\mb I\}$ produces a natural isomorphism
\[
\Map_{\proP{\cal C}}(\mb I, F_{\proP{\cal C}}(Y,Z)) \cong \Map_{\proP{\cal
    C}}(Y,Z).
\]
\end{rmk}

\begin{rmk}
  The functor $F_{\proP{\cal C}}$ does not generally act as an
  internal function object, in large part due to the presence of the
  colimit in the definition.
\end{rmk}

\subsection{Homotopical properties of function objects}

We continue the assumptions of Section~\ref{sec:function-objects} on
${\cal C}$.

\begin{prop}
\label{prop:functionobjects}
Suppose that filtered colimits preserve fibrations, represent homotopy
colimits, and commute with finite limits in ${\cal C}$.  Then the
function object $F_{\proP{\cal C}}(X,Y)$ satisfies the following
properties.
\begin{enumerate}
\item For a fixed $Y \in \proP{\cal C}$ and a cofiltered index
  category $I$, $F_{\proP{\cal C}}(-,Y)$ takes finite colimits in
  ${\cal C}^I$ to finite limits in $\proP{\cal C}$.
\item For a fixed $X \in \proP{\cal C}$ and a cofiltered index
  category $I$, $F_{\proP{\cal C}}(X,-)$ takes finite limits in
  ${\cal C}^I$ to finite limits in $\proP{\cal C}$.
\item The function object satisfies an SM7 axiom:  for any strict
  cofibration $i\co A \cofb B$ and strict fibration $p\co X
  \fibr Y$ in $\proP{\cal C}$, the induced map
\[
F_{\proP{\cal C}}(B,X) \to F_{\proP{\cal C}}(B,Y)
\mathop\times_{F_{\proP{\cal C}}(A,Y)} F_{\proP{\cal C}}(A,X)
\]
is a fibration, which is a strict equivalence if either $i$ or $p$ is.
\end{enumerate}
\end{prop}

\begin{proof}
  \begin{enumerate} 
  \item For a finite diagram $J \to {\cal C}^I$, the colimit as a
    diagram of pro-objects is computed by the colimit in ${\cal C}^I$
    \cite[Appendix~4.2]{artin-mazur-etale}.  By assumption, the natural
    morphisms
    \[
    \colim_\alpha \lim_j F_{\cal C}(x^j_\alpha, y_\beta) \to \lim_j \colim_\alpha
    F_{\cal C}(x^j_\alpha, y_\beta)
    \]
    are isomorphisms for all $\beta$, so we find that the natural map
    $F_{\proP{\cal C}}(\colim_j X^j, Y) \to \lim_j F_{\proP{\cal
        C}}(X^j,Y)$ is an isomorphism.
  \item The proof of this item is identical to that of the previous one.
  \item We note that the statement is preserved by retracts in $p$,
    and so we may assume that $p\co X \fibr Y$ is a special fibration.
    We can choose level representations for $i$ and $p$ with several
    properties:
    \begin{itemize}
    \item the map $i$ is a levelwise cofibration $\{a_\alpha \cofb b_\alpha\}_\alpha$,
    \item the map $i$ is a levelwise acyclic cofibration if $i$ is a
      strict equivalence \cite[4.13]{isaksen-strict},
    \item the map $p$ is indexed by a cofinite directed set,
    \item the maps from $x_\beta$ to $M_\beta = y_\beta
      \mathop\times_{\lim_{\gamma < \beta} y_\gamma} (\lim_{\gamma < \beta}
      x_\gamma)$ defined by $p$ are fibrations, and
    \item the fibrations $x_\beta \fibr M_\beta$ are weak equivalences if $p$
      is a strict equivalence \cite[4.14]{isaksen-strict}.
    \end{itemize}
    The pushout product axiom in ${\cal C}$ is equivalent to the
    internal SM7 axiom.  Hence for all $\alpha$ and $\beta$, we find
    that the map
    \[
      F_{\cal C}(b_\alpha,x_\beta) \to F_{\cal C}(b_\alpha,M_\beta)
      \mathop\times_{F_{\cal C}(a_\alpha,M_\beta)} F_{\cal C}(a_\alpha,x_\beta)
    \]
    is a fibration, and is a weak equivalence if $i$ or $p$ is a
    strict equivalence.  Using the fact that $F_{\cal C}$ preserves limits in the 
    target variable, this says that the natural map
    \[
      F_{\cal C}(b_\alpha, x_\beta) \to Z_{\alpha,\beta}
      \mathop\times_{\lim_{\gamma < \beta}
        Z_{\alpha, \gamma}} \bigl(\lim_{\gamma < \beta} F_{\cal C}(b_\alpha, x_\gamma)\bigr)
    \]
    is a fibration which is trivial if $i$ or $p$ is, where
    \[
    Z_{\alpha,\gamma} = F_{\cal C}(b_\alpha, y_\gamma)
    \mathop\times_{F_{\cal C}(a_\alpha, y_\gamma)} F_{\cal
      C}(a_\alpha, x_\gamma)
    \]
    is the component of the fiber product in degree $\gamma$.

    Taking colimits in $\alpha$, which commutes with the fiber product
    and preserves fibrations by assumption, we obtain
    a level representation of the map
    \[
    F_{\proP{\cal C}}(B,X) \to F_{\proP{\cal C}}(B,Y)
    \mathop\times_{F_{\mathrm{pro}\textrm{-}{\cal C}}(A,Y)} F_{\proP{\cal C}}(A,X)
    \]
    by a special fibration.  Since filtered colimits represent
    homotopy colimits, they preserve weak equivalences, and so this is
    a levelwise equivalence if $i$ or $p$ is a strict equivalence, as
    desired.\qedhere
  \end{enumerate}
\end{proof}

\begin{rmk}
  This actually proves that the SM7 map of $F_{\proP{\cal C}}$ already
  provides a special fibration or special acyclic fibration if the
  original map $p$ is a special fibration or special acyclic
  fibration.
\end{rmk}

\begin{cor}
  Under the assumptions of Proposition~\ref{prop:functionobjects}, we
  have the following consequences.
\begin{enumerate}
\item For fibrant $Y \in \proP{\cal C}$, $F_{\proP{\cal C}}(-,Y)$ preserves weak
  equivalences between cofibrant objects.
\item For cofibrant $X \in \proP{\cal C}$, $F_{\proP{\cal C}}(X,-)$ preserves
  weak equivalences between fibrant objects.
\item The functor $F_{\proP{\cal C}}(-,-)$ descends to a well-defined
  weak function object $F^{weak}(-,-)$ for the homotopy category of
  cofibrant-fibrant objects of $\proP{\cal C}$.
\end{enumerate}
\end{cor}

\begin{proof}
  By Ken Brown's lemma \cite[1.1.12]{hovey-modelcategories}, to prove
  the first item it suffices to prove that $F_{\proP{\cal C}}(-,Y)$ takes
  acyclic cofibrations to weak equivalences.  This follows by applying
  the SM7 property to an acyclic cofibration $X \to X'$ and the
  fibration $Y \to *$.
  
  The second item follows exactly as in the previous case by applying
  the SM7 property to an acyclic fibration.  The final item is then a
  direct consequence.
\end{proof}

The following proposition allows us to gain homotopical control on
function objects from the associated pro-objects in the homotopy
category.

\begin{prop}
\label{prop:functionhomotopytype}
  Suppose that filtered colimits preserve fibrations, represent
  homotopy colimits, and commute with finite limits in ${\cal C}$.
  For levelwise cofibrant $X$ and levelwise fibrant $Y$ in
  $\proP{\cal C}$ with fibrant replacement $Y'$, the map
  $F_{\proP{\cal C}}(X,Y) \to F_{\proP{\cal C}}(X,Y')$ is a weak
  equivalence.  The representative $F_{\proP{\cal C}}(X,Y')$ for the
  homotopically correct weak function object $F^{weak}(X,Y')$ is a
  representative for the homotopy type
\[
\{\hocolim_\alpha \mb RF_{\cal C}(x_\alpha, y_\beta)\}_\beta.
\]
\end{prop}

\begin{proof}
  This argument closely follows {\cite[5.3]{fausk-isaksen-tmodel}}.
  By assumption $X$ is strict cofibrant, and by reindexing we may
  assume that $Y$ is indexed by a cofinite directed set $I$ and still
  levelwise fibrant.  The index category $I$ is a Reedy category, so
  we may choose a Reedy fibrant replacement $Y \to Y'$ which is
  a levelwise weak equivalence so that the maps $y'_\beta \to
  \lim_{\gamma < \beta} y'_\gamma$ are fibrations.  In particular,
  $y'_\beta$ is always fibrant.

  The levelwise properties imply that the function objects $F_{\cal
    C}(x_\alpha, y_\beta)$ are representatives for the derived
  function objects, and that the maps $F_{\cal C}(x_\alpha, y_\beta)
  \to F_{\cal C}(x_\alpha, y'_\beta)$ are weak equivalences.

  As in the proof of Proposition~\ref{prop:functionobjects},  the
  homotopically correct function object $\{\colim_\alpha F_{\cal
    C}(x_\alpha, y'_\beta)\}_\beta$ is levelwise equivalent to
  $\{\colim_\alpha F_{\cal C}(x_\alpha, y_\beta)\}_\beta$.  Since
  colimits represent homotopy colimits, we obtain the desired result.
\end{proof}

\subsection{Pro-dualizable objects}

Continuing the assumptions of Section~\ref{sec:function-objects}, we
now begin to study dualizability.

\begin{defn}
  For an object $x \in ho({\cal C})$, the {\em dual} $Dx$ is
  the function object $\mb RF_{\cal C}(x,\mb I)$.  The map $Dx
  \owedge^{\mb L} x \to \mb I$ is the {\em evaluation pairing}.

Given $y \in ho({\cal C})$, the adjoint to the map $Dx \owedge^{\mb L} x \owedge^{\mb
    L} y \to y$ is a natural transformation $Dx \owedge^{\mb L} y \to
  \mb RF_{\cal C}(x,y)$.  The object $x \in ho({\cal C})$ is {\em
    dualizable} if this map is a natural isomorphism of functors on
  $ho({\cal C})$.

  An object $X \in ho(\proP{\cal C})$ is {\em pro-dualizable} if it is
  isomorphic in the homotopy category to a cofiltered diagram of
  objects whose images in the homotopy category are dualizable.
\end{defn}

\begin{rmk}
  We follow \cite{hovey-strickland-moravaktheory} in using the term
  {\em dualizable}, rather than the term {\em strongly dualizable}
  from \cite{hovey-palmieri-strickland-axiomatic}.
\end{rmk}

The following are immediate consequences of the definitions.

\begin{prop}
  The unit $\mb I$ is dualizable.  Dualizable objects are closed under
  the tensor in $ho({\cal C})$, and pro-dualizable objects are closed
  under the levelwise tensor in $ho(\proP{\cal C})$.
\end{prop}

Suppose that the unit object $\mb I$ is {\em compact}, in the sense
that the functor $\mb R\Map_{\cal C}(\mb I,-)$ commutes with
filtered homotopy colimits.  Then the natural equivalence
\[
\mb R\Map_{\cal C}(x,y) \simeq \mb R\Map_{\cal C}(\mb I, Dx
\owedge^{\mb L} y)
\]
implies that $\mb R\Map_{\cal C}(x,-)$ commutes with filtered homotopy
colimits.  We then have the following result, which is similar in
spirit to the earlier results \cite[B.3,~(2)]{bauer-formal} and
\cite[9.15]{fausk-equivariant}.

\begin{prop}
\label{prop:dualizableadjoint}
Suppose that filtered colimits preserve fibrations, represent homotopy
colimits, and commute with finite limits in ${\cal C}$.  In addition,
suppose that $\mb R\Map_{\cal C}(\mb I,-)$ commutes with filtered
homotopy colimits.  Let $X \in ho(\proP{\cal C})$ be pro-dualizable.
Then, for any $Y$ and $Z$ in $ho(\proP{\cal C})$, the weak function
object $F^{weak}$ provides an adjoint to the map $X \owedge^{\mb L} Y
\to Z$ (\ref{def:weakfunction}).
\end{prop}

\begin{proof}
  Without loss of generality, we can lift $X$ to a pro-object
  represented by a diagram which is levelwise cofibrant and
  dualizable.  Similarly, we choose lifts of $Y$ to a strict cofibrant
  diagram and $Z$ to a special fibrant diagram, which in particular is
  levelwise fibrant.

  Combining Propositions~\ref{prop:maphomotopytype} and 
  \ref{prop:functionhomotopytype}, the natural
  transformation
\[
\Map_{\proP{\cal C}}(X \owedge^{\mb L} Y, Z) \to \Map_{\proP{\cal
    C}}(X,F^{weak}(Y,Z))
\] 
is naturally represented by the map of homotopy types
\[
\holim_\gamma \hocolim_{\alpha,\beta} \mb R\Map_{\cal C}(x_\alpha
\owedge y_\beta, z_\gamma) \to \holim_\gamma \hocolim_\alpha \mb
R\Map_{\cal C}(x_\alpha, \hocolim_\beta \mb RF_{\cal C}(y_\beta,
z_\gamma)).
\]
As $\mb R\Map_{\cal C}(x_\alpha,-)$ commutes with filtered homotopy
colimits, this reduces to the adjunction $\mb R\Map_{\cal C}(x_\alpha
\owedge y_\beta, z_\gamma) \cong \mb R\Map_{\cal C}(x_\alpha, F_{\cal
  C}(y_\beta, z_\gamma))$.
\end{proof}

Combining this with Theorem~\ref{thm:rigidlifting}, we have the
following result.

\begin{thm}
\label{thm:rigiddualizable}
Suppose that filtered colimits preserve fibrations, represent homotopy
colimits, and commute with finite limits in ${\cal C}$.  In addition,
suppose that $\mb R\Map_{\cal C}(\mb I,-)$ commutes with filtered
homotopy colimits.  Let $\eta\co X \to Y$ in $ho(\proP{\cal C})$ be a
rigid map between pro-dualizable objects.  If $X$ is an algebra over a
cofibrant operad ${\cal O}$, then there exists an ${\cal O}$-algebra
structure on $Y$, compatible with $\eta$, which is unique up to
homotopy.
\end{thm}

\section{Symmetric spectra and filtered colimits}
\label{sec:modelcats}

In this section, we verify several conditions on model categories of
interest in this paper.  In particular, we show that the ``base
category'' of symmetric spectra is a proper operadic model category,
and hence, has an associated model category of pro-objects that is
operadic. Also, we show that this base model structure satisfies
several required assumptions from Section~\ref{sec:pro-objects}.


We write $\Sp$ for the category of symmetric spectra in simplicial
sets described in \cite{hovey-shipley-smith-symmetric}.  For $R$ a
ring object in $\Sp$, we write $\Sp_R$ for the category of
$R$-modules.  We will follow \cite{schwede-book} (which uses the term
``absolute flat stable'') in using the term {\em flat stable} model
structure for what is called the $R$-model structure in
\cite{shipley-convenient}.

The properties of filtered colimits preserving fibrations, preserving
weak equivalences, and commuting with finite limits are true in the 
category of simplicial sets, and are inherited by several categories
based on diagrams of them.

\begin{prop}
  \label{prop:symspec-work}
  Let $R$ be a commutative ring object in $\Sp$.  Under the flat
  stable model structure, the category $\Sp_R$ is a proper monoidal
  model category under $\smsh{R}$ with a compatible simplicial
  enrichment and cofibrant unit.  In this category, filtered colimits
  represent homotopy colimits, commute with finite limits, and
  preserve fibrations.  Mapping spaces out of $R$ commute with
  filtered homotopy colimits.
\end{prop}

\begin{rmk}
As in \cite[2.8]{shipley-convenient}, the identity functor is a
Quillen equivalence between the ordinary stable model structure and
the flat stable model structure.  By
Remark~\ref{rmk:easytensoradjunction}, the flat stable and ordinary
stable model structures are essentially equivalent for considering
operadic structures.
\end{rmk}

\begin{proof}
  By \cite[2.6, 2.7]{shipley-convenient}, the flat stable model
  structure makes $\Sp_R$ a proper monoidal model category.
  The simplicial enrichment is tensored and cotensored, with the
  tensor compatible by definition.
  
  The adjoint of the pushout product axiom implies that the
  internal function objects $F_R(-,-)$ obey an SM7 axiom in the flat
  stable model structure on $R$-modules.  To conclude that the
  simplicial enrichment satisfies the SM7 axiom, it suffices to note
  that $\Map_R(X,Y)$ is the degree zero portion of the function
  spectrum $F_R(X,Y)$, and that the functor taking an $R$-module to
  its degree zero portion is a right Quillen functor (with adjoint $R
  \wedge (-)$).

  As this model category is cofibrantly generated, and the
  generating cofibrations and acyclic cofibrations $A \cofb B$ have
  source and target which are compact, filtered colimits automatically
  preserve fibrations (cf. \cite[5.3.1]{behrens-davis-fixedpoints}).

  Filtered colimits and finite limits in $\Sp_R$ are formed levelwise
  in the category of pointed simplicial sets.  In particular, filtered
  colimits commute with finite limits and preserve level
  equivalences.

  Given a diagram $\{X_i\}_i$ in $\Sp_R$ indexed by a cofinite
  directed set $I$, the homotopy colimit is the left derived functor
  of colimit, and is formed by taking the colimit of a cofibrant
  replacement $\{X'_i\}_i \to \{X_i\}_i$ in the projective model
  structure on $I$-diagrams.  As the classes of cofibrations coincide
  in the projective model structures on $I$-diagrams for the flat
  level and flat stable model structures on $\Sp_R$, and similarly for
  the classes of acyclic fibrations, the above cofibrant replacement
  is, objectwise, a level equivalence.  The natural map $\colim_i X'_i
  \to \colim_i X_i$ is then a filtered colimit of level equivalences;
  it is therefore a level equivalence, and hence a stable equivalence.
  
  The functor $\Map_R(R,-)$ is the zeroth space functor, which
  commutes with homotopy colimits.
\end{proof}

\section{Towers of Moore spectra}
\label{sec:mainthm}

In this section we will show that there are towers of Moore spectra
admitting an $E_\infty$ structure.  We will deduce from this the
existence of $E_\infty$ structures on pro-spectrum lifts of the
$K(n)$-local spheres, the telescopic ($T(n)$-local) spheres, and the
Morava $E$-theories $E(k, \Gamma)$.

Throughout this section we fix a prime $p$.

\begin{thm}[{\cite[4.22]{hovey-strickland-moravaktheory}}]
For any integer $n \geq 1$, there is a tower $\{M_I\}_I$ of generalized Moore
spectra of type $n$ under $\mb S$ such that, for all finite spectra
$Z$ of type greater than or equal to $n$, the natural map
\[
Z \to \{Z \wedge M_I\}_I
\]
is an isomorphism of pro-objects in the homotopy
category of spectra. 

Any two such towers are isomorphic as pro-objects in the
homotopy category.
\end{thm}

We will refer to any such tower $\{M_I\}_I$ of generalized Moore
spectra under $\mb S$ as a {\em Moore tower}.

\begin{cor}
For any Moore tower $\{M_I\}_I$, the unit map $\mb S \to
\{M_I\}_I$ is rigid {\rm (\ref{def:rigid})} in the homotopy category
of pro-spectra.
\end{cor}

\begin{proof}
By Proposition~\ref{prop:functionhomotopytype} it suffices to show
that for any $I$, the natural map
\[
\hocolim_J F(M_J, M_I) \to M_I
\]
is an equivalence.

The dual of $M_I$ is still finite of type $n$, and so the natural map
\[
DM_I \to \{DM_I \wedge M_J\}_J
\]
becomes an isomorphism of pro-objects in the homotopy category.
Taking duals, we find that
\[
\{F(M_J, M_I)\}_J \to M_I
\]
becomes an isomorphism of ind-objects in the homotopy category.  In
particular, this map expresses the domain as being ind-constant in the
homotopy category, and so the induced map
\[
\hocolim_J F(M_J, M_I) \to M_I
\]
is a weak equivalence, as desired.
\end{proof}

\begin{thm}
\label{thm:mainthm}
Any Moore tower admits the structure of an $E_\infty$-algebra. 

For $n \geq 1$, let $K(n)$ and $T(n)$ denote a Morava $K$-theory of height $n$ and the
mapping telescope of a $v_n$-self-map of a type $n$ complex.  Then the
$K(n)$-local and $T(n)$-local spheres lift to the $E_\infty$-algebras 
$\{L_{K(n)}{\mb S} \wedge M_I\}_I$ and $\{L_{T(n)}{\mb S} \wedge M_I\}_I$ 
in the category of pro-spectra.
\end{thm}

\begin{proof}
As Moore towers are pro-dualizable, application of
Corollary~\ref{cor:rigid-einfty} and
Proposition~\ref{prop:dualizableadjoint} yields the first statement.

By smashing a Moore tower with the constant pro-objects 
$L_{K(n)}\mb S$ and $L_{T(n)}\mb S$, each of which is 
an $E_\infty$-algebra (since localizations of $\mb S$ are 
$E_\infty$-algebras in spectra), and noting that the inverse limit is still the
$K(n)$-local or $T(n)$-local sphere, we obtain the second statement. 
\end{proof}

Since $E(k, \Gamma)$ is $K(n)$-local and an $E_\infty$-algebra in
spectra, the argument for the second part of the above theorem gives
the following result.

\begin{cor}
Let $n \geq 1$, let $k$ be any perfect field of characteristic $p$, and let 
$\Gamma$ be any height $n$ formal group law over $k$. Then 
$E(k, \Gamma)$ lifts to the $E_\infty$-algebra $\{E(k, \Gamma) \wedge M_I\}_I$ 
in the category of pro-spectra, functorially in $(k,\Gamma)$.
\end{cor}

\section{Nilpotent completions}
\label{sec:completions}

In this section, we roughly follow Bousfield \cite[\S 5]
{bousfield-spectralocalization} in defining nilpotent resolutions,
though we have been influenced by \cite{carlsson-completions} and
\cite{baker-lazarev-adamss}.

To recap assumptions, in this section the category ${\cal C}$ is
\begin{itemize}
\item an operadic model category,
\item whose monoidal structure is closed,
\item whose underlying model category is proper, and
\item whose filtered colimits preserve fibrations, realize homotopy
  colimits, and commute with finite limits.
\end{itemize}
Finally, we now add the assumption that 
\begin{itemize}
\item the underlying model category is stable.
\end{itemize}
As a result, $ho({\cal C})$ has the structure of a tensor triangulated
category.

The main example in mind is the category of modules over a
commutative symmetric ring spectrum.

\begin{defn}[{cf. \cite[3.7]{bousfield-spectralocalization}}]
  Suppose $E$ is an object in $ho({\cal C})$.
  The category $Nil(E)$ of {\em $E$-nilpotent objects} is the smallest
  subcategory of $ho({\cal C})$ containing $E$ which is closed under
  isomorphisms, cofiber sequences, retracts, and tensoring with
  arbitrary objects of $ho({\cal C})$.
\end{defn}

In other words, $Nil(E)$ is the thick tensor ideal of $ho({\cal C})$
generated by $E$, or equivalently the thick subcategory generated by
objects of the form $(E \owedge^{\mb L} x)$ for $x \in ho({\cal C})$.

\begin{defn}[{cf. \cite[5.6]{bousfield-spectralocalization}}]
For an element $E$ in $ho({\cal C})$, an {\em $E$-nilpotent resolution}
of $y \in ho({\cal C})$ is a tower $\{w_s\}_s$ of objects under $y$ such
that
\begin{enumerate}
\item $w_s$ is in $Nil(E)$ for all $s \geq 0$, and
\item for any $E$-nilpotent object $z$, the map $\hocolim_s \mb
  RF_{\cal C}(w_s,z) \to \mb RF_{\cal C}(y,z)$ is a weak
  equivalence.
\end{enumerate}
\end{defn}

\begin{rmk}
The second condition is preserved by cofiber sequences, isomorphisms,
and retracts in $z$, and so it suffices to show it for objects of the
form $(E \owedge^{\mb L} x)$ for $x \in ho({\cal C})$.
\end{rmk}

\begin{rmk}
Suppose that the category $ho({\cal C})$ has a collection of dualizable
generators $p_i$.  To check that a tower is an $E$-nilpotent
resolution, it suffices to check that the maps
\[
[p_i, \hocolim_s \mb RF_{\cal C}(w_s,z)] \to [p_i, \mb RF_{\cal C}(y,z)]
\]
are isomorphisms, or equivalently that the maps
\[
\colim_s\,[w_s,Dp_i \owedge^{\mb L} z] \to [y,Dp_i \owedge^{\mb L} z]
\]
are isomorphisms.  However, since $z$ is $E$-nilpotent, so is $Dp_i
\owedge^{\mb L} z$, and therefore it suffices to check that the map
\[
\colim_s\,[w_s,z] \to [y,z]
\]
is an isomorphism for all $E$-nilpotent $z$ as in Bousfield's
definition.
\end{rmk}

Given an $E$-nilpotent resolution $\{w_s\}_s$ of an object $y$, we can
lift it to a map in $\proP{\cal C}$ from the constant pro-object $y$ to
a representing tower $\{w_s\}_s$.  We will casually refer to
a map of towers $\{y\} \to \{w_s\}_s$ in ${\cal C}$ as an $E$-nilpotent
resolution if the domain is constant with value $y$ and the image of
the range in the homotopy category is an $E$-nilpotent resolution of
$y$.  We will view ${\cal C}$ as embedded in the category of towers in
${\cal C}$ so that we may abuse notation by writing this as a map $y \to
\{w_s\}_s$.

\begin{prop}
\label{prop:rigidresolution}
If $\eta\co y \to W$ and $\eta'\co y \to W'$ are two $E$-nilpotent
resolutions of $y$ in ${\cal C}$, then the map $F^{weak}(W,W') \to
F^{weak}(y,W')$ is a strict equivalence.  In particular, the
map $\eta$ in $\proP{\cal C}$ is rigid.
\end{prop}

\begin{proof}
We may assume without loss of generality that $y$ is cofibrant and
that the towers $W = \{w_s\}_s$ and $W' = \{w'_t\}_t$ are levelwise
cofibrant-fibrant in $\cal C$. By Proposition
\ref{prop:functionhomotopytype}, the map $F^{weak}(W,W') \to
F^{weak}(y,W')$ is homotopy equivalent to a map of pro-objects
\[
  \{\hocolim_s F_{\cal C}(w_s,w'_t)\}_t \to \{F_{\cal C}(y,w'_t)\}_t.
\]
As each $w'_t$ is $E$-nilpotent, the maps $\hocolim_s F_{\cal C}(w_s,
w'_t) \to F_{\cal C}(y,w'_t)$ are weak equivalences.  Therefore, the
associated map of towers $F^{weak}(W,W') \to F^{weak}(y,W')$ is a
levelwise equivalence.
\end{proof}

Any two $E$-nilpotent resolutions are therefore pro-isomorphic in the
homotopy category $ho(\proP{\cal C})$.  We therefore will often 
casually refer to a map in the pro-category $y \to \comp{y}_E$, from
the constant object $y$ to an $E$-nilpotent resolution, as {\em the}
$E$-nilpotent completion of $y$.

\begin{defn}[{cf. \cite[4.8]{hovey-strickland-moravaktheory}}]
A {\em $\mu$-ring} is an object $E \in ho({\cal C})$ equipped with a map
$\mb I \to E$ and a multiplication $E \owedge^{\mb L} E \to E$ which
is left unital.
\end{defn}

\begin{prop}
\label{prop:dualizableresolution}
If the object $E$ is a $\mu$-ring in $ho({\cal C})$ and $y \in {\cal
C}$, there exists an $E$-nilpotent resolution of $y$.  If $E$ and
$y$ are dualizable, then there exists a resolution which is
pro-dualizable.
\end{prop}

\begin{proof}
We apply the standard techniques to provide a canonical ``Adams
resolution'' of $y$ as follows.  First form a fiber sequence $J \to
\mb I \to E$ in $ho({\cal C})$.  The maps
\[
J^{\owedge^{\mb L} (n+1)} \to \mb I \owedge^{\mb L} J^{\owedge^{\mb L}
  n} \cong J^{\owedge^{\mb L} n}
\]
construct a tower of tensor powers of $J$.  We then define
$\mb I/J^n$ as the cofiber of the composite map $J^{\owedge^{\mb L} n}
\to \mb I^{\owedge^{\mb L} n} \cong \mb I$.

For an arbitrary object $y$, to show that the tower
\[
\{(\mb I/J^n) \owedge^{\mb L} y\}_n
\]
is an $E$-nilpotent resolution of $y$, it suffices to show that for
any object $x \in ho({\cal C})$ the map
\[
\hocolim_n \mb RF_{\cal C}((\mb I/J^n) \owedge^{\mb L} y, E
\owedge^{\mb L} x) \to \mb RF_{\cal C}(y,E \owedge^{\mb L} x)
\]
is a weak equivalence.  Taking fibers, it suffices to show that
\[
\hocolim_n \mb RF_{\cal C}(J^{\owedge^{\mb L}n} \owedge^{\mb L} y, E
\owedge^{\mb L} x)
\]
is trivial.  However, the $\mu$-ring structure on $E$ implies that any
map from $J^{\owedge^{\mb L}n} \owedge^{\mb L} y$ to
$E \owedge^{\mb L} x$ automatically lifts to a map from $E
\owedge^{\mb L} J^{\owedge^{\mb L}n} \owedge^{\mb L} y$, and thus
restricts to the trivial map from $J^{\owedge^{\mb L}(n+1)} \owedge^{\mb L} y$.

Dualizable objects are closed under cofiber sequences and tensor
products \cite[2.1.3]{hovey-palmieri-strickland-axiomatic}, and so if
$E$ and $y$ are dualizable this tower is pro-dualizable.
\end{proof}

\begin{thm}
\label{thm:completionisalgebra}
Suppose that $E$ is a dualizable $\mu$-ring in $ho({\cal C})$, and $y
\in {\cal C}$ is an algebra over a cofibrant operad ${\cal O}$. Then
there exists a unique ${\cal O}$-algebra structure on the
$E$-nilpotent completion $\comp{y}_E$ which is compatible with $y$.
\end{thm}

\begin{proof}
  This is obtained by applying Theorem~\ref{thm:rigidlifting}, the
  hypotheses of which are verified by
  Propositions~\ref{prop:dualizableresolution},
  \ref{prop:rigidresolution}, and \ref{prop:dualizableadjoint}.
\end{proof}

\begin{rmk}
  When Theorem \ref{thm:completionisalgebra} is applied to the
  category of modules over $E(k, \Gamma)$ (where $E(k, \Gamma)$ is any
  Morava $E$-theory, as defined in Section \ref{sec:intro}) with
  dualizable $\mu$-ring given by the associated $2$-periodic Morava
  $K$-theory, we recover the $E_\infty$-structure on $E(k, \Gamma)$ in
  the category of pro-spectra.  However, this construction does not
  respect the action of the extended Morava stabilizer group $G(k,
  \Gamma) = \Aut_{\scriptscriptstyle{{\cal{FG}}}}(k, \Gamma)$,
  the automorphism group of $(k, \Gamma)$ in the category
  ${\cal{FG}}$. One could also apply this method to the smash product
  of a $\mu$-ring with $E(k,\Gamma)$, $L_{E(n)} \mb S$, $L_{K(n)} \mb
  S$, or $L_{T(n)}\mb S$.
\end{rmk}

\bibliography{masterbib}

\end{document}